\documentclass[11pt,reqno]{amsart}

\usepackage{CBstyle}
\usepackage[T1]{fontenc}
\usepackage[english]{babel}

\usepackage[margin=1.1in,footskip=0.2in]{geometry}

\title{The Manin--Drinfeld theorem and the rationality of Rademacher symbols}

\author{Claire Burrin}
\address{ETH Z\"urich, D-MATH, Zurich, Switzerland}
\email{claire.burrin@math.ethz.ch}

\date{\today}

\begin{document}
\maketitle

\begin{abstract}
For any noncocompact Fuchsian group $\G$, we show that periods of the canonical differential of the third kind associated to residue divisors of cusps are expressed in terms of Rademacher symbols for $\Gamma$ --- generalizations of periods appearing in the classical theory of modular forms. This result provides a relation between Rademacher symbols and the famous theorem of Manin and Drinfeld. More precisely, Fuchsian groups whose Rademacher symbols are rational-valued verify the statement of Manin--Drinfeld. We then establish the rationality of Rademacher symbols for various families of Fuchsian groups.
%
\end{abstract}

\section{{\bf Introduction}}

Let $\Gamma$ be a noncocompact Fuchsian group. Then $\Gamma\backslash\h$ is an open Riemann surface, and its compactification $X$ can be seen as a smooth algebraic curve. Fixing a base point $x\in X$, we have an embedding  $X\hookrightarrow J(X)$ into its Jacobian $J(X)$. Manin and Mumford conjectured that for any such embedding, there are only finitely many torsion points in the image. The conjecture was first confirmed by Raynaud \cite{Raynaud1983}. 

If $\G$ is a congruence subgroup of $\SL_2(\Z)$, Drinfeld \cite{Drinfeld1973}, generalizing a result of Manin \cite{Manin1972}, provided significant examples of these finitely many torsion points on modular curves; the subgroup $C(\G)\subset J(X)$ spanned by the images of the cusps of $\G$ is contained in $J(X)^{\rm tor}$. Their proofs used Hecke operators, limiting the reach of the statement to modular curves. In this note, we observe a new connection between the Manin--Drinfeld theorem and Rademacher symbols, which are ubiquitous invariants appearing in connection to Dedekind sums \cite{Rademacher}, class numbers of real quadratic fields \cite{Meyer1957,Zagier1975}, aspects of Atiyah--Bott--Singer index theory \cite{Atiyah1987}, and linking numbers of knots \cite{GhysICM,DukeImamogluToth2017}. On the basis of this relation, we give a new proof of the Manin--Drinfeld theorem and examine for which other families of curves it might hold.

As is customary, we reframe this discussion in terms of divisors. Recall that a divisor is a formal sum 
$$
D = \sum a_i x_i,
$$
where $x_i\in X$, and $a_i$ are integers, only finitely many of which are non-zero. The degree of the divisor $D$ is the sum $\sum a_i$, and divisors of degree zero form a group. Let $D$ be a divisor of degree zero. By the Abel--Jacobi theorem, there is an isomorphism between divisor classes (with respect to linear equivalence) of degree zero and the Jacobian $J(X)$ modulo periods. 
By Riemann--Roch, there is a differential $\omega_D$ of the third kind (i.e., a meromorphic differential) whose residual divisor is $D$, and this differential can be made unique; see, e.g., Theorem 5.3 in \cite{LangAbelian}. Then $D$ is torsion if and only if the periods of $\omega_D$ are in $2\pi i\Q$.

\begin{Thm}\label{thm:1}
Let $\Gamma$ be a noncocompact Fuchsian group with $h$ inequivalent cusps
 $\frak{a}_1,\dots,\frak{a}_h$. Let $D=\sum m_i (\frak{a}_i)$ be a divisor of degree zero, and let $\omega_D$ be the canonical differential of the third kind associated to the residue divisor $D$. Then for every hyperbolic element $\g\in\G$ of positive trace,
$$
\int_\gamma \omega_D = 2\pi i\left(m_1\Psi_{\frak{a}_1}(\gamma)+\dots+m_h \Psi_{\frak{a}_h}(\gamma)\right),
$$
where $\Psi_\frak{a}$ is the Rademacher symbol on $\Gamma$ for the cusp $\frak{a}$. Moreover, if all Rademacher symbols on $\G$ are rational-valued, then $C(\G)\subset J(X)^{\rm tor}$. 
\end{Thm}

We first describe Rademacher symbols as they appear in the classical theory of modular forms and elliptic functions.  Dedekind showed that the transformation-theory of the $\eta$-function
\begin{align}\label{eta}
\eta(z) = q^{\tfrac{1}{24}}\prod_{n\geq1} (1-q^n),
\end{align}
($q=e^{2\pi iz}$), under the linear fractional action of $\g=\bsm a&b\\ c&d\esm\in\SL_2(\Z)$ on the upper half-plane $\h$  is encoded by the Dedekind symbol
$$
\Phi(\gamma)\ =\ \begin{dcases} \frac{b}{d} & {\rm if }\ c=0,\\
\frac{a+d}{c}-12\sign(c) \cdot s(a,c) & {\rm if }\ c\neq0.
\end{dcases} 
$$
where $s(a,c)$ are Dedekind sums --- arithmetic sums depending only on $a/c$, and obeying a reciprocity law that makes them easy to compute via the Euclidean algorithm. The same reciprocity law also allows to show that $\Phi(\g)$ is always an integer. These classical facts are exposed in the monograph of Rademacher's lecture notes on the subject; see \cite{Rademacher}. In his study of Dedekind sums, Rademacher introduced the associated function 
\begin{align}\label{Rademacher}
\Psi(\gamma)=\Phi(\gamma)-3\sign(c(a+d)),
\end{align}
which has the remarkable feature of being conjugacy class invariant. In fact, on hyperbolic elements, it can be realized as the period
$$
\int_\g E_2(z)dz = \Psi(\g),
$$
where $E_2$ is Hecke's modified Eisenstein series of weight 2 (see \secref{sec:proof} and \cite{DukeImamogluToth2018} for further historical references). 

This is for the classical theory. Over the last 50 years, analogues of these various functions have been constructed for other noncocompact Fuchsian groups. A standard construction builds on extending the celebrated first limit formula of Kronecker, which classically can be used to deduce the transformation-theory of $\eta$ and determine the Dedekind sums. The classical first limit formula of Kronecker expresses the constant term in the Laurent expansion of the non-holomorphic Eisenstein series $E(z,s)$ for $\PSL_2(\Z)$ at its pole $s=1$. More generally, to each cusp $\frak{a}$ of a general noncocompact Fuchsian group, one has an associated non-holomorphic Eisenstein series $E_\fa(z,s)$. On the basis of Selberg's study of their analytic properties, one obtains a Kronecker limit formula in this extended setting, and this limit formula can be used to define a generalized Dedekind symbol $\Phi_\fa$. Such constructions have been extensively studied in recent years \cite{Goldstein1973,Burrin2017,JorgensonO'SullivanSmajlovic2020}. Much less attention has been given the associated Rademacher symbol $\Psi_\fa:\G\to\R$, which we define in \secref{sec:Rademacher}. In \lemref{lemma}, we show that $\Psi_\fa$ can again be realized as the hyperbolic period 
$$
\int_\g E_{2,\fa}(z)dz = \Psi_\fa(\g).
$$ 

The Fuchsian Rademacher symbols $\Psi_\fa$ are a priori real-valued. In the face of \thmref{thm:1}, we do not expect Fuchsian Rademacher symbols to be rational-valued in general. In fact, Rohrlich observed that for some non-congruence subgroups of $\Gamma(2)$ appearing in connection to the Fermat curve, $C(\Gamma)$ must be infinite (see p.~198 in \cite{KubertLangBook}) and thus the associated Rademacher symbols can not all be rational. 
 Nevertheless, for certain groups, it is possible to show that the image is in fact rational through rigorous arithmetic computations. This is the case for the Rademacher symbol $\Psi_\infty^\G$ when $\G$ is any principal congruence group $\G(N)$ or Hecke congruence group $\G_0(N)$, following the works of \cite{Driencourt1983}, \cite{Takada1986}, and \cite{Vassileva}. By a theorem of Helling \cite{Helling1966}, every noncocompact arithmetic subgroup of $\SL_2(\R)$ is conjugate to a subgroup of $\G_0(N)^+$, the group obtained by adding all Atkin--Lehner involutions to $\G_0(N)$, for some squarefree $N$. (See \secref{sec:rationality} for an explicit parametrization.) Moreover, each maximal noncocompact arithmetic group (under inclusion) is a conjugate of $\G_0(N)^+$ for some squarefree $N$. (The converse does not hold; see \cite{Baily1987}.) As noted by Helling, the groups $\G_0(N)^+$ have a unique cusp, which can be taken at the point at $\infty$. The 44 groups $\G_0(N)^+$ of genus 0 famously appear in connection to monstruous moonshine.

\begin{Thm}\label{thm:2}
Rademacher symbols are rational for all genus 0 noncomcompact Fuchsian groups, for all Helling groups $\G_0(N)^+$, and for all congruence subgroups of $\SL_2(\Z)$.
\end{Thm}

In particular, we recover a new proof of the Manin--Drinfeld theorem. By a theorem of Takeuchi \cite{Takeuchi1977}, we know that among triangle groups, only finitely many (up to conjugation) are arithmetic; it is interesting to note that the phenomenon described by Manin--Drinfeld does not restrict to arithmetic groups. The proof of \thmref{thm:2} builds on certain relations between Rademacher symbols for $\G_1$ and $\G$, when $\G_1<\G$ is a normal subgroup of finite index; see \propref{proposition}.

\subsection{Acknowledgments.} This note was inspired by the beautiful paper of Murty and Ramakrishnan \cite{MurtyRamakrishnan1987} on a similar connection between the Manin--Drinfeld theorem and generalized Ramanujan sums. The author thanks Jay Jorgenson for many discussions on connections between Manin--Drinfeld and Dedekind symbols of various types, and Asbj\o rn Nordentoft for pointing out a mistake in a previous version of this manuscript.

\section{{\bf Rademacher symbols for cusps}}\label{sec:Rademacher}

The classical monographs of Shimura \cite{Shimura} and Kubota \cite{Kubota} and the more recent book of Iwaniec \cite{Iwaniec} serve as references for facts cited in this section on Fuchsian groups and Eisenstein series. 

\subsection{Cusps, Eisenstein series}

Let $\frak{a}$ be a cusp for $\Gamma$, and let $\Gamma_\frak{a}$ be the stabilizer subgroup of $\frak{a}$, i.e., $\Gamma_\frak{a}=\{\gamma\in\Gamma:\gamma\frak{a}=\frak{a}\}$. The group $\Gamma_\frak{a}$ is infinite cyclic. We will denote by $\g_\fa$ the cyclic generator of $\G_\fa$. We say that two cusps $\frak{a}$, $\frak{b}$ are equivalent if $\frak{b}=\gamma\frak{a}$ for some $\gamma\in\Gamma$. If $\frak{a}$ and $\frak{b}$ are equivalent, then $\Gamma_\frak{a}=\Gamma_\frak{b}$. 

We may choose a scaling transformation $\sigma_\frak{a}\in\PSL(2,\R)$ such that 
\begin{align*}
\sigma_\frak{a}(\infty)&=\frak{a}\qquad\text{and}\qquad
\sigma_\frak{a}^{-1}\Gamma_\frak{a}\sigma_\frak{a} = \pm\bpm 1&\Z\\&1\epm.
\end{align*}
Such a choice of $\sigma_\frak{a}$ is only unique up to right multiplication by an element of $\pm\bsm 1&\R\\0&1\esm$. Conjugating $\Gamma$ by $\sigma_\frak{a}$ provides a group with a cusp at infinity of width 1, and facilitates computation.

The group $\Gamma$ acts discontinuously on the upper half-plane $\h$ by fractional linear transformation $\pm\bsm a&b\\ c&d\esm:z\mapsto\tfrac{az+b}{cz+d}$. For each $z\in\h$, the Eisenstein series for the cusp $\frak{a}$ is defined by
$$
E_\frak{a}(z,s) =\sum_{\gamma\in\Gamma_\frak{a}\backslash\Gamma} \im(\sigma_\frak{a}^{-1}\gamma z)^s,
$$
where $s$ is a complex parameter controlling the convergence of the infinite series; the series $E_\frak{a}(z,s)$ converges absolutely and uniformly on compact subsets when $\re(s)>1$. Beyond this half-plane, Eisenstein series admit a meromorphic continuation to the whole complex plane, as famously proved by Selberg. The definition of  $E_\frak{a}(z,s)$ does not depend on the particular choice of scaling $\sigma_\frak{a}$. Moreover, equivalent cusps yield identical Eisenstein series. 

To understand the behavior of the Eisenstein series $E_\frak{a}$ at the (not necessarily distinct) cusp $\frak{b}$ of $\Gamma$, one considers the series $E_\frak{a}(\sigma_\frak{b}z,s)$, for which we have the relation
$$
E_\frak{a}(\sigma_\frak{b}z,s) = E_\frak{a}(\gamma_\frak{b} \sigma_\frak{b}z,s) = E_\frak{a}(\sigma_\frak{b}(z+1),s)
$$
and the resulting Fourier expansion is explicitly given by
$$
E_\fa(\sigma_\fb z,s) = \delta_\fab y^s +\frac{\sqrt\pi \G(s-1/2)}{\G(s)}\varphi_\fab(s) y^{1-s} +\frac{2\pi^s \sqrt{y}}{\G(s)} \sum_{n\neq0} |n|^{s-1/2}K_{s-1/2}(2\pi|n|y)\varphi_\fab(n,s) e(nx),
$$
where $\delta_\frak{ab}=1$ if the cusps $\frak{a}$ and $\frak{b}$ are in the same $\Gamma$-orbit, and 0 otherwise, where $\G(s)$ denotes the classical $\G$-function, where $K_s(z)$ is the $K$-Bessel function (or, modified Bessel function of the second kind), where $e(z)=e^{2\pi i z}$, and where $\varphi_\fab(s)$ and $\varphi_\fab(n,s)$ are the Dirichlet series
\begin{align*}
\varphi_\fab(s) &= \sum_{c>0} c^{-2s} \#\{ d\in[0,c): \bsm *&*\\ c&d\esm\in\sigma_\fa^{-1}\G\sigma_\fb\}, \\ 
\varphi_\fab(n,s) &= \sum_{c>0} c^{-2s} \sum_{\substack{0\leq d<c\\ \bsm *&*\\ c&d\esm\in \sigma_\fa^{-1}\G\sigma_\fb}} e\left(n\frac{d}{c}\right).
\end{align*}
These Dirichlet series do not depend on a particular representative for the cusps $\fa$, $\fb$, but $\varphi_\fab(n,s)$ depends on the particular choice of the scaling transformation $\sigma_\fb$ up to a unitary multiplicative factor. Indeed, if we replace $\sigma_\fb$ by $\sigma_\fb n_x$, $n_x=\pm\bsm 1& x\\ 0&1\esm$, then $\varphi_\fab(n,s)$ is replaced by $\varphi_\fab(n,s)e(nx)$. Both Dirichlet series are absolutely convergent on $\re(s)>1$. 

On the vertical line $\re(s)=1$, the Eisenstein series are holomorphic, save for a simple pole at $s=1$. Examining the local behavior of its Fourier coefficients, we note that the $K$-Bessel function is entire as a function of $s$ --- in fact, at $s=1$, we have $2(|n|y)^{1/2} K_{1/2}(2\pi|n|y)e(nx) =  e^{-2\pi|n|y}e(nx)$, see \cite[p.~227]{Iwaniec} ---, and that the Dirichlet series $\varphi_\fab(n,s)$ are holomorphic along the vertical line $\re(s)=1$. At $s=1$, $\varphi_\fab(n,1)$ is real valued, hence $\varphi_\fab(n,1)=\varphi_\fab(-n,1)$, and satisfies the estimate $|\varphi_\fab(n,1)|\ll n^{1+\eps}$, for any $\eps>0$; see \cite[Theorem 1.1]{JorgensonO'Sullivan2005}. Finally, the Dirichlet series defining $\varphi_\fab(s)$ has a simple pole at $s=1$ of residue $(\pi V)^{-1}$, where $V$ denotes the hyperbolic volume of any fundamental domain for $\G$. Thus $\varphi_\fab(s)$ has a Laurent expansion at $s=1$ of the form
$$
\varphi_\fab(s) = \frac{(\pi V)^{-1}}{s-1} +\sum_{n\geq0} c_n (s-1)^n.
$$

\subsection{Kronecker's first limit formula}

An extension (see \cite[Theorem 3-1]{Goldstein1973}) of Kronecker's classical first limit formula gives the constant coefficient in the Laurent series for $E_\frak{a}(\sigma_\frak{b}z,s)$ at the simple pole at $s=1$. Using the Fourier expansion of the Eisenstein series and letting $s\to1^+$, we immediately obtain the corresponding formal Laurent series at $s=1$, given by
\begin{align*}
\lim_{s\to 1^+} \left(E_\frak{a}(\sigma_\frak{b}z,s) - \frac{V^{-1}}{s-1}\right)\ &=\ c_0 - V^{-1}\log y +\delta_\fab y +\sum_{n>0} \pi \varphi_\fab(n,1) e(nz) + \sum_{n<0} \pi \varphi_\fab(n,1) e(n\overline{z})\\
& =\ c_0 -V^{-1}\log y - \re\left(\delta_\fab iz - 2\pi \sum_{n>0} \varphi_\fab(n,1)e(nz)\right).
\end{align*}
We denote the term in parenthesis by $f_{\fa,\sigma_\fb}(z)$ and record that $f_{\fa,\gs_\fb n_x}(z)=f_{\fa,\gs_\fb}(n_x z) -\delta_\fab ix$. As a function of $z$, $f_{\fa,\gs_\fb}$ is holomorphic; this establishes the validity of the limit formula.

Specializing to $\G=\PSL_2(\Z)$, we recover the classical first limit formula of Kronecker as follows. We have
\begin{align*}
E(z,s) &= \frac12 \sum_{\substack{c,d\in\Z\\ (c,d)=1}} \frac{y^s}{|cz+d|^{2s}},\qquad \varphi(s) = \sum_{c\geq1} \frac{\phi(c)}{c^{2s}},\\
\varphi(n,s) &= \sum_{c\geq1} c^{-2s} \sum_{\substack{d=1\\ (c,d)=1}}^c e\left(n\frac{d}{c}\right) = \sum_{c\geq1} \frac{1}{c^{2s}}\sum_{\delta|(c,n)}\delta\mu\left(\frac{c}{\delta}\right) = \frac{1}{\zeta(2s)}\sum_{\delta|n} \frac{1}{\delta},
\end{align*}
where $\phi$ is Euler's totient function, and we have used Kluyven's identity for Ramanujan sums in the last line. Using that $\zeta(2)=\tfrac{\pi^2}{6}$, 
$$
f(z) = \frac{12}{\pi}\left(\frac{\pi iz}{12}-\sum_{m,n>0} \frac{q^{mn}}{m}\right),
$$
and the term in parenthesis is a branch of the logarithm of Dedekind's $\eta$-function, defined by \eqref{eta}.

\subsection{Dedekind and Rademacher symbols}
We introduce two auxiliary functions on $\G$; the inner automorphism $\tau_{\sigma_\fb}(\gamma)=\sigma_\fb^{-1}\gamma \sigma_\fb$, and the automorphic cocycle $j(\gamma,z)=cz+d$. The latter is a multiplicative cocycle: for any pair $\g_1,\g_2\in\G$, we have $j(\g_1\g_2,z)=j(\g_1,\g_2 z)j(\g_2,z)$. The limit formula derived above implies the relation
$$
\re f_{\fa,\gs_\fb}(\tau_\fb(\g)z) = V^{-1} \ln|j(\tau_\fb(\g),z)|^2 + \re f_{\fa,\gs_\fb}(z)
$$
for each $\g\in\G$. We choose $\log z$ to be the principal branch of logarithm, i.e., $\log z = \ln|z|+i\arg(z)$ with $\arg(z)\in(-\pi,\pi]$. Then $\log(-j(\g,z)^2)$ is well defined, and $\re\left( \sign(c)^2 \log(-j(\g,z)^2)\right) = \ln|j(\g,z)|^2$. 
Therefore, the holomorphic function
\begin{align}\label{F}
F_{\frak{a},\gs_\fb}(\g;z): = f_{\frak{a},\gs_\fb}(\tau_{\gs_\fb}(\gamma)z) -f_{\frak{a},\gs_\fb}(z) - V^{-1}\sign(c_{\tau_{\gs_\fb}(\g)})^2 \log\left(-j(\tau_{\gs_\fb}(\g),z)^2\right)
\end{align}
has trivial real part, and thus by the Open Mapping Theorem, it must be a constant function of $z$, i.e., $F_{\fa,\gs_\fb}(\g;z)=F_{\fa,\gs_\fb}(\g)$. As a result, it is also independent of the particular choice of the scaling transformation $\gs_\fb$, since
$$
F_{\fa,\gs_\fb n_x}(\g) = f_{\fa,\gs_b}(\tau_{\gs_\fb}(\g)n_x z) - f_{\fa,\gs_\fb}(n_x z)  -V^{-1}\sign(c_{\tau_{\gs_\fb}(\g)})^2\log\left(-j(\tau_{\gs_\fb}(\g),n_x z)^2\right) = F_{\fa,\gs_\fb}(\g).
$$
We henceforth denote this function $F_\fab$. We can now define the Dedekind and Rademacher symbols attached to the cusp $\fa$. Set
\begin{align*}
\Phi_\frak{ab}(\gamma) &\coloneqq-i F_\frak{ab}(\gamma),\\
 \Phi_\frak{a}(\gamma) &\coloneqq\Phi_\frak{aa}(\gamma),\\
\Psi_\frak{a}(\gamma) &\coloneqq  \Phi_\frak{a}(\gamma) - \pi V^{-1}\sign( c(a+d)),
\end{align*}
with $\bsm a&b\\ c&d\esm=\tau_{\gs_\fa}(\g)$.
When $\Gamma=\PSL(2,\Z)$, we recover the classical Dedekind and Rademacher symbols $\Phi$ and $\Psi$; see Chapters 4A-C in \cite{Rademacher}. A detailed study of the properties of (a slight modification of) the Dedekind symbols $\Phi_\fab$ appears in \cite{JorgensonO'SullivanSmajlovic2020}. We only record here that, following the analysis in pp.~52--53 of \cite{Rademacher}, 
\begin{align}\label{qm}
\Phi_\fa(\g_1\g_2) = \Phi_\fa(\g_1) + \Phi_\fa(\g_2) - \pi V^{-1}\sign(c_1 c_2 c_3)
\end{align}
for any $\g_1$, $\g_2\in\G$, with $\bsm *&*\\ c_1&*\esm=\tau_{\gs_\fa}(\g_1)$, $\bsm *&*\\ c_2&*\esm=\tau_{\gs_\fa}(\g_2)$, and $\bsm *&*\\ c_3&*\esm=\tau_{\gs_\fa}(\g_1\g_2)$. The following relations follow immediately: 
$\Phi_\fa(I) = \Psi_\fa(I) = 0$, $\Phi_\fa(-\g) = \Phi_\fa(\g)$, $\Psi_\fa(-\g) = \Psi_\fa(\g)$, 
$\Phi_\fa(\g^{-1}) = -\Psi_\fa(\g)$, $\Psi_\fa(\g^{-1})=-\Psi_\fa(\g).$

\begin{prop}\label{prop:3}
Rademacher symbols are rational on elliptic and parabolic motions. 
\end{prop}

\begin{proof}
Let $\g$ be a parabolic element. Recall that $\G_\fb$ is an infinite cyclic group and denote its generator by $\g_\fb$. Then $\g=\g_\fb^m$ for some $m\in\N$.  We first observe that $\Phi_\fab(\g)= -i\left( f_{\fa,\gs_\fb}(z+m) - f_{\fa,\gs_\fb}(z)\right) = \delta_\fab \cdot m$. Hence $\Phi_\fb$ is an integer on $\G_\fb$. Suppose now that $\fa$ and $\fb$ are inequivalent cusps. Since $E_\fa(\gs_\fa z,s) = E_\fa(\gs_\fb z',s)$ with $z'=\gs_\fb^{-1}\gs_\fa z$, the Kronecker limit formula implies the relation
$$
f_{\fa,\gs_\fb}(z')-f_{\fa,\gs_\fa}(z) = V^{-1}\log\left(-j(\gs_\fb^{-1}\gs_\fa,z)^2\right) + \text{constant}.
$$
Then
\begin{align*}
\Phi_\fa(\g) &= -i\left( f_{\fa,\gs_\fa}(\tau_{\gs_\fa}(\g)z) - f_{\fa,\gs_\fa}(z) - V^{-1}\log\left(-j(\tau_{\gs_\fa}(\g),z)^2\right)\right)\\
&=
-i\left( f_{\fa,\gs_\fb}(z'+m)-V^{-1}\log(-j(\gs_\fb^{-1}\gs_\fa,\tau_{\gs_\fa}(\g_\fb)z)^2\right) - f_{\fa,\gs_\fb}(z')\\
& \qquad  + V^{-1} \log\left(-j(\gs_\fb^{-1}\gs_\fa,z)^2\right) - V^{-1}\log\left(-j(\tau_{\gs_\fa}(\g),z)^2\right))\\
& =
\Phi_\fab(\g) +2iV^{-1}\left( \log j(\gs_\fb^{-1}\gs_\fa,\tau_{\gs_\fa}(\g)z) + \log j(\tau_{\gs_\fa}(\g),z) - \log j(\gs_\fb^{-1}\gs_\fa,z) + i\frac{\pi}{2} A_{\fab}(\g)\right)\\
& = 
\delta_\fab\cdot m +2iV^{-1}\left(2\pi i B_\fab(\g)+i\frac{\pi}{2}A_\fab(\g)\right),
\end{align*}
where $A_\fab$ and $B_\fab$ are integer-valued. 
Looking at the Gauss--Bonnet formula for $V$ (see, e.g., \cite[p.~46]{Iwaniec}), we have $\pi V^{-1}\in\Q$ and this concludes. Let now $\g$ be an elliptic element of order $m$. Applying \eqref{qm} recursively, we have
\begin{align}\label{recursion}
0 = \Phi_\fa(\g^m) & = m\Phi_\fa(\g) - \pi V^{-1} \sum_{k=1}^{m-1} \sign( c_\g c_{\g^{k}} c_{\g^{k+1}}) 
\end{align}
and this shows that $\Phi_\fa(\g)\in\Q$. Since $\G$ is finitely generated and the transformation formula \eqref{qm} for $\Phi_\fa$ depends only on the rational constant $\pi V^{-1}$, we conclude that genus 0 noncocompact Fuchsian groups have rational-valued Rademacher symbols.
\end{proof}

For certain congruence subgroups, it is possible to deduce an explicit expression for $\Phi_\fa$, following the method of proof of Dedekind. For example, using the Fourier expansion of the Eisenstein series for $\G(N)$ and the cusp at $\infty$, the function $f:=f_{\infty,\sigma_\infty}$, $(\sigma_\infty=I$), arising from the Kronecker limit formula is explicitly given by \cite[Theorem 1]{Takada1986}
$$
f(z) = iz + 4V^{-1}_{\G(N)}\sum_{j=0}^{N-1}\log \prod_{\substack{m\geq1\\ m\equiv j (N)}} (1-q^{mN})
$$
and (following the method of Dedekind) the corresponding Dedekind symbol is\footnote{ We have the following dictionary between our notation and Takada's: $\pi\alpha_N = V^{-1}_{\G(N)}$, $f= 4\pi\alpha_N \log\eta_N$
, $\Phi(\g) = 4\pi \alpha_N(S_N(\g)-\tfrac14 \sign(c))$ (for $\g=\bsm *&*\\ c&*\esm$ with $c\neq0$).} \cite[Theorem 2]{Takada1986}
$$
\Phi^{\G(N)}_\infty(\g) = \begin{dcases} \frac{b}{d} & \text{ if } c=0, \\
\frac{a+d}{c} - \frac{4\pi V_{\G(N)}^{-1}}{|c|N}\sum_{j=1}^{|c|N-1} j C_{N,j} \left(\!\left(\frac{Naj}{c}\right)\!\right) & \text{ if } c\neq0,\end{dcases}
$$
where $(\!(x)\!)= x - \lfloor x\rfloor -\tfrac12$ if $x\not\in\Z$ and 0 otherwise, and
$$
C_{N,j} = \frac{\pi^2}{6}\prod_{p\mid N}(1-p^{-2})\sum_{\substack{a=1\\ (a,N)=1}}^N \sum_{\substack{n\geq1\\ na\equiv 1 (N)}} \frac{\mu(n)}{n^2} \cos\frac{2\pi aj}{N}.
$$

If $N=2$, then $C_{2,j}=\cos(\pi j)\in\{\pm 1\}$ and the Dedekind symbol is rational-valued. If $N>2$, Takada shows that $C_{N,j}$ lies in the cyclotomic field $\Q(\zeta_N,\zeta_{\phi(N)})$, where $\phi(N)$ is the Euler totient function. This implies that $\Phi^{\G(N)}_\infty$ --- and in turns $\Psi^{\G(N)}_\infty$ --- are rational-valued for all $N$'s.

\section{{\bf Proof of Theorem 1}}\label{sec:proof}

For each cusp $\frak{a}$, consider the modified Eisenstein series
\begin{align*}
E_{2,\frak{a}}(z,s) \coloneqq\ 2i\frac{\partial}{\partial z} E_\frak{a}(z,s) = \left(\frac{\partial}{\partial y} +i\frac{\partial}{\partial x}\right) E_\frak{a}(z,s).
\end{align*}
Since the residue of $E_\frak{a}(z,s)$ at the simple pole at $s=1$ is constant in $z$, this modified Eisenstein series is regular at $s=1$, and we set $E_{2,\frak{a}}(z)\coloneqq E_{2,\frak{a}}(z,1)$. Explicitly, we have
\begin{align*}
E_{2,\fa}(z) 
&= \lim_{s\to1^+} sy^{s-1} \sum_{\g\in\G_\fa\bk\G} j(\gs_\fa^{-1}\g,z)^{-2}|j(\gs_\fa^{-1}\g,z)|^{-2s+2}.
\end{align*} 
Hence this is a real-analytic Eisenstein series of weight 2; for principal congruence groups, this recovers Hecke's construction.
 The Kronecker limit formula for $E_\fa(\gs_\fb z,s)$ yields the Fourier expansion of $E_{2,\fa}(\gs_\fb z)$, which is  given by
\begin{align*}
E_{2,\fa}(\gs_\fb z) \gs_\fb'(z) &= 2i \ddz\left( c_0 - V^{-1}\log y +\delta_\fab y +\sum_{n>0} \pi \varphi_\fab(n,1)e(nz) +\sum_{n<0}\pi \varphi_\fab (n,1)e(n\overline{z})\right)\\
&= -\frac{V^{-1}}{y} +\delta_\fab -4\pi^2\sum_{n>0}n\varphi_\fab(n,1)e(nz)\\
& = 
-\frac{V^{-1}}{y}+\frac{1}{i}\frac{d}{dz}f_{\fa,\gs_\fb}(z).
\end{align*}

The relation \eqref{qm} implies that $\Psi_\fa(\g^{-1}) = -\Psi_\fa(\g)$ and $\Psi_\fa(-\g)=\Psi_\fa(\g)$. Thus up to replacing $\g$ by $\pm\g^{\pm1}$, we may assume that for $\g=\bsm a&b \\c &d\esm$, we have $c>0$ and $a+d>0$.

\begin{lm}\label{lemma}
Let $\g\in\G$ be a hyperbolic element of positive trace. Then
$$
\int_{\g} E_{2,\fa}(z)dz = \Psi_\fa(\g).
$$
\end{lm}

\begin{proof}
Fix a base point $z_0\in\h$ on the geodesic axis of $\g$. Using the above Fourier expansion,
\begin{align*}
\int_{z_0}^{\g z_0} E_{2,\fa}(z) dz &= \int_{\gs^{-1}_\fa z_0}^{\gs^{-1}_\fa \g z_0} E_{2,\fa}(\gs_\fa z) \gs_\fa'(z) dz\\
&= -V^{-1} \int_{\gs^{-1}_\fa z_0}^{\gs^{-1}_\fa \g z_0} \frac{dz}{y} +\frac{1}{i}\left( f_{\fa,\gs_\fa}(\tau_{\gs_\fa}(\g)\gs_\fa^{-1}z_0)-f_{\fa,\gs_\fa}(\gs_\fa^{-1}z_0)\right).
\end{align*}
For any $g=\bsm *&*\\c&d\esm\in\SL_2(\R)$, we have 
$$
\frac{dz}{y}\circ g - \frac{dz}{y} = \frac{|cz+d|^2-(cz+d)^2}{(cz+d)^2 y}dz = -\frac{2ic}{cz+d} dz = -2i\ d \log j(g,z)
$$
and this shows that 
$$
\int_{\gs_\fa^{-1}z_0}^{\tau_{\gs_\fa}(\g)\gs_\fa^{-1}z_0} \frac{dz}{y} = -2i \log j(\tau_{\gs_\fa}(\g),\gs_\fa^{-1}z_0).
$$
On the other hand,
\begin{align*}
f_{\fa,\gs_\fa}(\tau_{\gs_\fa}(\g)\gs_\fa^{-1}z_0)-f_{\fa,\gs_\fa}(\gs_\fa^{-1}z_0) =
i\Phi_\fa(\g)+V^{-1} \sign(c_{\tau_{\gs_\fa}(\g)})^2 \log\left(-j(\tau_{\gs_\fa}(\g),\gs_\fa^{-1}z_0)^2\right)
\end{align*}
Since $\tr(\g)>2$, we have $c_{\tau_{\gs_\fa}(\g)}\neq0$ and
\begin{align*}
\log\left(-j(\tau_{\gs_\fa}(\g),\gs_\fa^{-1}z_0)^2\right) &= 2\log\left(\frac{j(\tau_{\gs_\fa}(\g),\gs_\fa^{-1}z_0)}{i \sign(c_{\tau_{\gs_\fa}(\g)})}\right)= 2\log j(\tau_{\gs_\fa}(\g),\gs_\fa^{-1}z_0)
- i\pi \sign(c_{\tau_{\gs_\fa}(\g)}).
\end{align*}
We conclude that 
$$
\int_{z_0}^{\g z_0} E_{2,\fa}(z)dz = \Phi_\fa(\g) - \pi V^{-1}\sign(c_{\tau_{\gs_\fa}(\g)}),
$$
which coincides with the definition of the Rademacher symbol $\Psi_\fa(\g)$ for $\g$ of positive trace.
\end{proof}

Let $\fa_1,\fa_2,\dots,\fa_h$ denote the inequivalent cusps of $\Gamma$. For each integer tuple ${\bf m}=(m_1,\dots, m_h)$ such that 
$
\sum m_i = 0,
$
define
$$
E_{2,\fm}(z) \coloneqq \sum_{i=1}^{h} m_i E_{2,\frak{a}_i}(z).
$$
Then $E_{2,\fm}(z)$ has the Fourier expansion
$$
E_{2,\fm}(\sigma_\frak{b} z)\gs'_\fb(z)
= m_\fb -4\pi^2\sum_{n\geq1} n \left( \sum_{i=1}^h m_i\varphi_{\fa_i \fb}(n) \right)e(nz),
$$
from which it can be seen that $E_{2,\fm}(z)$ is a holomorphic modular form of weight 2. In particular, $E_{2,\fm}$ induces a holomorphic differential on $X$ and we conclude our proof of \thmref{thm:1} with the following result of Scholl; see Proposition 2 in \cite{Scholl1986}.

\begin{prop} Let ${\bf m}=(m_1,\dots, m_h)$. In the notation above,
$$
2\pi i E_{2,\fm}(z)dz =\omega_D
$$
which is the canonical differential of the third kind associated to the residue divisor $D=\sum m_i(\frak{a}_i).$
\end{prop}

Let 
$$
I_{\bf m}(\g) := 2\pi i \int_{z_0}^{\g z_0} E_{2,{\bf m}}(z)dz
$$
and note that the integral does not depend on the particular choice of $z_0\in\h$. In particular, $I_{{\bf m}}:\G\to\R$ is a group homomorphism and $I_{{\bf m}}(\g)=0$ if $\g$ is elliptic. Considering the Fourier expansion of $E_{2,{\bf m}}$, we see that if $\g=\g_\fb^k$, where $\g_\fb$ denotes the generator of the parabolic subgroup $\G_\fb$, then $I_{\bf m}(\g) = k\cdot m_\fb$. We conclude that the periods of $\omega_D$ are in $2\pi i\Q$ if and only if all Rademacher symbols for $\G$ are rational-valued on hyperbolic elements.

\section{{\bf Proof of Theorem 2}}\label{sec:rationality}

\begin{lm}\label{lm:simplification}
If $\G_1<\G$ is a subgroup of finite index, then each Rademacher symbol $\Psi_\fa$ is rational if its restriction to $\G_1$, i.e., $\Psi_\fa\vert_{\G_1}$ is rational on hyperbolic elements.
\end{lm}

\begin{proof}
By (\ref{qm}), we have that for any $\g_1$, $\g_2\in \G$,
$$
\pi^{-1} V\Psi_\fa(\g_1\g_2) \equiv \pi^{-1} V(\Psi_\fa(\g_1) +\Psi_\fa(\g_2))\qquad (\text{mod }1),
$$
where we recall that $\pi^{-1} V\in\Q$. We write $\G=\cup_\tau \G_1\tau$ as a disjoint union of $\G_1$-cosets. Since $\G_1$ has finite index, there is some $n\in\N$ such that $\tau^n\in\G_1$. By the recursion formula (\ref{recursion}),
$$
\pi^{-1} V \Psi_\fa(\tau^n) \equiv n\pi^{-1} V \Psi_\fa(\tau)\qquad (\text{mod }1).
$$
Each $\g\in\G$ can be written as $\g=\g_1\tau$ for some $\g_1\in\G_1$ and some coset representative $\tau$, so that
$$
n\pi V^{-1} \Psi_\fa(\gamma) \equiv \pi V^{-1}(n\Psi_\fa(\g_1) + \Psi_\fa(\tau^n))\qquad (\text{mod }1).
$$
By \propref{prop:3}, Rademacher symbols are always rational on parabolic and elliptic elements, hence it suffices to check that $\Psi_\fa\vert_{\G_1}$ is rational on hyperbolic elements.
\end{proof}

It immediately follows that Rademacher symbols on genus 0 Fuchsian groups are rational. To obtain new classes of examples among Fuchsian groups of genus $g\geq1$, we build on the following relations along normal covers.

\begin{prop}\label{proposition}
Let $\G_1<\G$ be a normal subgroup of finite index, and let $\fa$ be a cusp of $\G_1$ with stabilizer subgroup $\G_{1,\fa}=\G_1\cap\G_\fa$. We denote the Rademacher symbols for the cusp at $\fa$ on $\G$ and $\G_1$ by $\Psi^\G_\fa$ and $\Psi^{\G_1}_\fa$ respectively. For each hyperbolic element $\gamma\in\Gamma_1$ of positive trace,
$$
\Psi_\fa^\G (\gamma) = \Psi^\G_\fa \vert_{\G_1}(\g) = \sum_{\tau\in\Gamma_1\backslash\Gamma} \Psi_\fa^{\Gamma_1}(\tau\gamma\tau^{-1}),
$$
and this formula does not depend on the particular choice of coset representatives. Moreover, if $\fa$ and $\fb$ are cusps of $\G_1$ that are $\G$-equivalent --- i.e., if there is some $\tau\in\G$ such that $\tau\fa=\fb$ --- then for each hyperbolic element $\g\in\G_1$ of positive trace,
$$
\Psi_\fa^{\G_1}(\gamma) = \Psi_\fb^{\G_1}(\tau\gamma\tau^{-1}).
$$
\end{prop}

\begin{proof}
Writing $\G=\cup_\tau \G_1\tau$ as a finite disjoint union of $\G_1$-cosets, we have the formal equality
$$
\sum_{\g\in\G_\fa\backslash\G} j(\gs_\fa^{-1}\g,z)^{-2} = \sum_\tau \sum_{\g\in\G_{1,\fa}\backslash\G_1} j(\gs_\fa^{-1}\g\tau,z)^{-2} =  \sum_\tau j(\tau,z)^{-2} \sum_{\g\in\G_{1,\fa}\backslash\G_1} j(\gs_\fa^{-1}\g,\tau z)^{-2}.
$$ 
Notice that this equality does not depend on the particular choice of coset representatives. Hence for each hyperbolic element $\g\in\G_1$ of positive trace, 
$$
\Psi^\G_\fa(\g) = \int_\g E^\G_{2,\fa}(z)dz = \sum_\tau \int_\g E^{\G_1}_{2,\fa}(\tau z)\tau'(z) dz = \sum_\tau \int_{\tau\g\tau^{-1}} E^{\G_1}_{2,\fa}(z)dz = \sum_\tau \Psi_\fa^{\G_1}(\tau\g\tau^{-1}).
$$
Suppose that $\tau\fa=\fb$, and set $\delta=\gs_\fb^{-1}\tau\gs_\fa.$ Since $\tau\gs_\fa\infty=\gs_\fb\infty$, we have $\delta=\bsm d &*\\ 0& d^{-1}\esm$ for some $d\neq0$. We claim that $d=\pm 1$. We have $\tau^{-1}\G_\fb\tau=\G_\fa$ and hence
$$
B:= \gs_\fa^{-1}\G_\fa \gs_\fa = (\tau\gs_\fa)^{-1}\G_\fb (\tau\gs_\fa) = \delta^{-1}(\gs_\fb^{-1}\G_\fb \gs_\fb) \delta = \delta^{-1}B\delta.
$$
Using that $B$ is generated by $\pm \bsm1 & 1\\0&1\esm$, a direction computation yields $d=\pm1$. We thus have the formal equality
\begin{align*}
\sum_{\g\in\G_{1,\fa}\backslash\G_1} j(\gs_\fa^{-1}\g,z)^{-2}&= \sum_{\g\in\G_{1,\fa}\backslash\G} j(\delta^{-1}\gs_\fb^{-1}\tau \g ,z)^{-2} =  j(\tau,z)^{-2} \sum_{\g\in\G_{1,\fa}\backslash\G} j(\gs_\fb^{-1}(\tau 
\g\tau^{-1}),\tau z)^{-2}\\
& =  j(\tau,z)^{-2}\sum_{\g\in\G_{1,\fb}\backslash\G_1} j(\gs_\fb^{-1}\g,\tau z)^{-2}
\end{align*}
where we used $\tau\G_{1,\fa} \tau^{-1}=\G_{1,\fb}$ and $\tau\G_1\tau^{-1}=\G_1$ for the last equality. Thus
$$
E^{\G_1}_{2,\fa}(z) = E^{\G_1}_{2,\fb}(\tau z) \tau'(z)
$$
and the equality of corresponding Rademacher symbols follows.
\end{proof}

Let $\G$ be a congruence subgroup of level $N$. By \lemref{lm:simplification} and \propref{proposition}, Rademacher symbols on $\G$ are rational-valued if Rademacher symbols on $\G(N)$ are rational-valued on hyperbolic elements. The computations of Takada (see the last paragraph of \secref{sec:Rademacher}) verify this for $\Psi_\infty^{\G(N)}$. Since $\G(N)$ is normal in $\G(1)$ and each cusp of $\G(N)$ is $\G(1)$-equivalent to the cusp at $\infty$, we conclude by the second part of \propref{proposition} that each Rademacher symbol on $\G(N)$ is rational-valued.

The groups $\Gamma_0(N)^+$ are obtained from $\G_0(N)$ by adding all (finitely many) Atkin--Lehner involutions and normalizing to matrices of determinant 1 by taking the quotient by a positive scalar. Explicitly, one has
$$
\Gamma_0(N)^+ = \left\{ e^{-1/2} \bpm a&b\\c&d\epm\in\SL_2(\R) :  a,b,c,d,e\in\Z,\  e\mid\mid N,\ e\mid a, d,\ N\mid c\right\},
$$
where $e\mid\mid N$ signifies that $e\mid N$ and $(e,\tfrac{N}{e})=1$. It is easily seen that the groups $\G_0(N)^+$ are commensurable to $\SL_2(\Z)$ but not necessarily conjugate to a subgroup of $\SL_2(\Z)$, that they have a single (inequivalent) cusp at $\infty$, and that $\G_0(N)^+$ contains $\G_0(N)$ as a normal subgroup of index $2^{\omega(N)}$, where $\omega(N)$ is the number of distinct prime divisors of $N$. Once again, by \lemref{lm:simplification} and \propref{proposition}, the Rademacher symbol on $\G_0(N)^+$ is rational if $\Psi_\infty^{\G_0(N)}$ is rational, which we have just established.

This proves \thmref{thm:2}. It would be very interesting to understand under which conditions the rationality of Rademacher symbols for a Fuchsian group $\G$ passes down to its normal subgroups $\G_1<\G$. For instance, finite-index normal subgroups of Fuchsian triangle groups are precisely the uniformizing groups of Galois Belyi curves, i.e.~curves that admit a Galois Belyi map, see \cite{ClarkVoight2019}, sometimes also called curves with many automorphisms, triangle curves, or quasiplatonic surfaces. There are only finitely many such curves for each given genus, but these include much studied examples such as the Fermat curves (for which Manin--Drinfeld is already established \cite{Rohrlich1977}).

\bibliographystyle{alpha}
\bibliography{biblio}

\end{document}